\documentclass[12pt,oneside]{amsart}

\usepackage{tkz-euclide,fullpage}
\tikzstyle{directed}=[postaction={decorate,decoration={markings,
    mark=at position .65 with {\arrow{stealth}}}}]
\tikzstyle{reverse directed}=[postaction={decorate,decoration={markings,
    mark=at position .5 with {\arrowreversed{stealth};}}}]
\usepackage{amssymb}
\usepackage{latexsym}
\usepackage{amsmath}
\usepackage{mathrsfs}
\usepackage[all]{xy}
\usepackage{enumerate}
\usepackage{amsthm}
\usepackage{psfrag}
\usepackage{ifthen}
\usepackage{comment}

\newtheorem{theorem}{Theorem}[section]
\newtheorem{lemma}[theorem]{Lemma}
\newtheorem{metalemma}[theorem]{Meta Lemma}
\newtheorem{proposition}[theorem]{Proposition}
\newtheorem{corollary}[theorem]{Corollary}

\theoremstyle{definition}
\newtheorem{definition}[theorem]{Definition}

\newtheorem{example}[theorem]{Example}

\newtheoremstyle{case}{}{}{}{}{}{:}{ }{}
\theoremstyle{case}
\newtheorem{case}{\bf{Case}}

\makeatletter
\def\blfootnote{\xdef\@thefnmark{}\@footnotetext}
\makeatother

\newcommand{\bel}[1]{\begin{equation}\label{#1}}
\newcommand{\ee}{\end{equation}}

\newcommand{\LBA}{\left\{\begin{array}}
\newcommand{\EAR}{\end{array}\right.}

\def\P{{\mathbf{P}}}
\def\NP{{\mathbf{NP}}}
\def\DP{{\mathbf{DP}}}

\newcommand{\gp}[1]{{\left\langle #1 \right\rangle}}
\newcommand{\gpr}[2]{{\left\langle #1 \mid #2 \right\rangle}}
\newcommand{\rb}[1]{{\left( #1 \right)}}
\newcommand{\abs}[1]{\left|{}#1\right|}

\newcommand{\Set}[2]{\left\{\, #1 \;\middle|\; #2 \,\right\}}

\def\MN{{\mathbb{N}}}
\def\MZ{{\mathbb{Z}}}
\def\MQ{{\mathbb{Q}}}
\def\MR{{\mathbb{R}}}
\def\BS{{\mathbf{BS}}}

\DeclareMathOperator{\Aut}{{Aut}}
\DeclareMathOperator{\TPART}{3PART}
\DeclareMathOperator{\PART}{PARTITION}

\title{Quadratic equations in metabelian Baumslag-Solitar groups}
\author{Richard Mandel}
\address{Department of Mathematics, University of the Basque Country, Bilbao, Spain}
\email{mandel.richard@ehu.eus}
\author{Alexander Ushakov}
\address{Department of Mathematical Sciences, Stevens Institute of Technology, Hoboken NJ 07030}
\email{aushakov@stevens.edu}
\thanks{The first author was partially supported by the Basque Government Grant IT1483-22 and the Spanish Government grants PID2019-107444GA-I00 and PID2020-117281GB-I00.}

\begin{document}

\maketitle

\begin{abstract}
For a finitely generated group $G$, the \emph{Diophantine problem} over $G$ is the algorithmic problem of deciding whether a given equation $W(z_1,z_2,\ldots,z_k) = 1$ (perhaps restricted to a fixed subclass of equations) has a solution in $G$. In this paper, we investigate the algorithmic complexity of the Diophantine problem for the class $\mathcal{C}$ of quadratic equations over the metabelian Baumslag-Solitar groups $\BS(1,n)$. We prove that this problem is $\NP$-complete whenever $n\neq \pm 1$, and determine the algorithmic complexity for various subclasses  (orientable, nonorientable etc.) of $\mathcal{C}$. 
\end{abstract}



\section{Introduction}

The study of equations over groups dates at least as far back as 1911, when Max Dehn proposed the word and conjugacy problems. However, the major work in this area began somewhat later; in particular, with Roger Lyndon's explicit description of the solutions to an arbitrary equation in one variable over a free group in 1960 (see \cite{Lyndon:1960(2)}), and it was not until the 80s that general decidability results were obtained. Namely, it was shown by Makanin \cite{Makanin:1982} and Razborov \cite{Razborov:1987} that the problem of deciding whether a general system of equations over a free group has a solution is decidable. This problem---to decide whether a given equation over a group $G$ has a solution---is known as the \emph{Diophantine problem} for $G$, denoted in this paper as $\DP(G)$. Since Makanin and Razborov, many results have been obtained establishing the decidability and, in some cases, the algorithmic complexity of the Diophantine problem for various classes of groups and equations (see, for instance, \cite{Grigorchuk-Lysenok:1992}, \cite{Casals-Ruiz_Kazachkov:2011b}, \cite{Lysenok-Ushakov:2015}, \cite{Lysenok-Miasnikov-Ushakov:2016}, \cite{Lysenok-Ushakov:2021}). For an overview of the subject (with an emphasis on solvable groups), see \cite{Roman'kov:2012}.

One class of equations over groups that has generated much interest is the class of \emph{quadratic} equations: equations where each variable $x$ appears twice (as either $x$ or $x^{-1}$). It was observed by Culler \cite{Culler:1981} and Schupp \cite{Schupp:1980} in the early 80s that such equations have an affinity with the theory of compact surfaces (for instance, via their associated van Kampen diagrams). This geometric point of view has led to many interesting results, particularly in the realm of quadratic equations over free groups. Some examples are the description of solution sets given by Grigorchuk and Kurchanov in \cite{Grigorchuk-Kurchanov:1992}, and the proof of $\NP$-completeness of the Diophantine problem due to Kharlampovich, Lysenok, Myasnikov and Touikan \cite{Kharlampovich-Lysenok-Myasnikov-Touikan:2010}. In this paper, we employ a combination of geometric and computational techniques to study the algorithmic complexity of quadratic equations over the metabelian Baumslag-Solitar groups $\BS(1,n)$.

\subsection{Baumslag-Solitar groups}The Baumslag-Solitar groups are the one-relator groups given by the presentation $\BS(m,n) = \langle a,t \ | \ t^{-1}a^mt = a^n \rangle$. First studied by Gilbert Baumslag and Donald Solitar in 1962 \cite{Baumslag_Solitar:1962}, they have proven to be an abundant source of counterexamples and unexpected properties. In particular, these groups provided the first known examples of one-relator non-Hopfian groups (groups which admit an isomorphic proper quotient\footnote{$\BS{(m,n)}$ is non-Hopfian whenever $m,n$ have different sets of prime divisors, e.g. $\BS(2,3)$.}), and have continued to inspire much fruitful research. It is a well-known fact that the groups $\BS{(1,n)}$ are metabelian for all $n$, and that these are the only solvable Baumslag-Solitar groups. Notably, the first-order theory of $\BS(1,n)$ is undecidable for $|n| \geq 2$, which is the case for any finitely generated virtually solvable group that is not virtually abelian (this was proved by G. A. Noskov in \cite{Noskov:1984}). Thus, determining which classes of equations and inequations are decidable over the metabelian Baumslag-Solitar groups (which are not even virtually nilpotent) appears to be an interesting and important endeavor.

The Diophantine problem for quadratic equations in $\BS(1,n)$  was recently shown to be decidable by Kharlampovich, Lopez and Miasnikov (see \cite{Kharlampovich-Lopez-Miasnikov:2020}). Numerous other algorithmic problems, including discrete optimization problems as well as decision problems, have also been recently studied in the metabelian Baumslag-Solitar groups. For instance, the knapsack problem was proved to be $\NP$-complete by M. Lohrey in \cite{Lohrey:2020}, whereas the Diophantine problem for systems of exponent equations\footnote{Equations of the form $g_1^{x_{i_1}}\cdots g_k^{x_{i_k}}$, with solutions sought in $\MN^k$.} was proved to be undecidable in the same paper. Decidability of the Diophantine problem for all groups which are virtually a direct product of hyperbolic groups, such as $\BS(n,\pm n)$, was proved by Ciobanu, Holt and Rees in \cite{Ciobanu-Holt-Rees:2020}, and in \cite{Duncan-Evetts-Holt-Rees:2023} the authors demonstrate that solution sets of certain families of equations over $\BS(1,n)$ form EDT0L languages. 

\subsection{Main results and outline}

In this paper we prove the following results.

\begin{theorem}\label{th:meta-general-complete}
The Diophantine problem for quadratic equations over $\BS(1,n)$ is $\NP$-complete whenever $|n|\neq 1$, and decidable in polynomial time otherwise.
\end{theorem}

\begin{theorem}\label{th:gen-np-hard}
Let $m,n\in \MZ$ be such that $|m| \neq |n|$. Then the Diophantine problem for quadratic equations over $\BS(m,n)$ is $\NP$-hard.
\end{theorem}

\begin{theorem}\label{th:solution-bound}
If a quadratic equation 
$W =  1$ over $\BS(1,n)$ has a solution,
then it has a solution of size $O(|W|^3)$.
\end{theorem}

We also provide a complete classification of the algorithmic complexity of the the Diophantine problem for the subclasses of \emph{orientable} (where each variable appears as both $x$ and $x^{-1}$) and \emph{nonorientable} (where at least one variable appears only as $x$ or only as $x^{-1}$) quadratic equations. Since the word problem for $\BS(1,n)$ is solvable in polynomial time (this is easily shown using Britton's lemma), Theorem \ref{th:solution-bound} provides an $\NP$-certificate in the form of the solution itself. Thus, Theorem \ref{th:meta-general-complete} is an immediate corollary of Theorems \ref{th:gen-np-hard} and \ref{th:solution-bound}. However, in Section \ref{se:metabelian-equations} we complete the proof of Theorem \ref{th:meta-general-complete} using different certificates which we find to be more convenient. These results are given in Corollaries \ref{th:meta-spherical-complete}, \ref{co:meta-nonorient-complete} and \ref{co:genus-geq2}.

Section \ref{se:prelim} contains background material and some general results on $\BS(1,n)$, and in Section \ref{se:np-hard} we prove Theorem \ref{th:gen-np-hard} (as Corollary \ref{co:spherical-NP-hard}). In what follows, $\DP$ will always refer to the Diophantine problem for the class of quadratic equations over $\BS(1,n)$ or (when explicitly stated) some subclass thereof.

Since quadratic equations over $\BS(1,1)\cong \MZ\times \MZ$ are easily seen to be trivial (i.e. equivalent to the word problem, with only trivial solutions being possible), we always assume that $n \neq 1$. In fact, we assume that $|n|\neq 1$ after Section \ref{se:polytime}, where it is proved that the quadratic Diophantine problem over $\BS(1,-1)$ is decidable in polynomial time. To prove Theorem \ref{th:gen-np-hard}, we exhibit a reduction of the 3-partition problem---a well-known $\NP$-complete problem---to $\DP(\BS(m,n))$. For both the 3-partition reduction and the complexity upper bound, we require a few results on the Diophantine problem for equations of the form 
$$q_1\alpha^{x_1}+\cdots +q_k\alpha^{x_k} = q, $$ with $q_i,q,\alpha \in \MQ$ and unknowns $x_i$. These results are stated in Section \ref{se:exp_eq}, and proofs may be found in \cite{Mandel-Ushakov:2022-b}. We assume a basic background in complexity theory and $\NP$-completeness, as may be found in \cite{Garey-Johnson:1979}.

\section{Preliminaries}\label{se:prelim}

\subsection{Quadratic equations over groups}
Let $F = F(Z)$ denote the free group on generators $Z = \{z_1,z_2,\ldots,z_k\}$. For a group $G$, an \emph{equation over $G$ with variables $z_1,\ldots,z_k$} is an expression of the form $W = 1$, where $W \in F*G$ (for convenience, we assume that all of the $z_i$ appear in $W$). If $W = c_0z_{i_1}c_1z_{i_2}\cdots z_{i_{l}}c_l$, with $c_j\in G$, then we refer to the nonidentity members of $\{c_0,\ldots,c_l\}$ as the \emph{constants} (or \emph{coefficients}) of $W$. We sometimes write $W(\bar{z})$, $W(z_1,\ldots,z_k)$ or $W(z_1,\ldots,z_k;c_1,\ldots,c_s)$ to emphasize the set of variables and constants in $W$.

A \emph{solution} to an equation $W(z_1,\ldots,z_k)=1$ over $G$ is a tuple $\bar{g} = (g_1,\ldots,g_k)\in G^k$ such that the mapping $z_i \mapsto g_i$ extends to a homomorphism $$\varphi:F*G\rightarrow G $$ which is the identity on $G$, and such that $\varphi(W) = 1$. We also write $W(g_1,\ldots,g_k)$ or $W(\bar{g})$ to denote the image of $W(z_1,\ldots,z_k)$ under $\varphi$. An equation $W=1$ is called \emph{quadratic} if each variable appears exactly twice (as either $z_i$ or $z_i^{-1}$).

\subsection{Size of an instance or solution of $\DP$}\label{se:size}

An instance of $\DP(\BS(1,n))$ is simply an element $W$ of $F(\bar{z})*\BS(1,n)$, given as a freely reduced word on the generators $z_i,$ $a$ and $t$, and the input size is defined to be the word length $|W|$ (we do not require $W$ to be in any normal form). The size of a solution $(g_1,\ldots,g_k)$ is defined as $\sum_{i=1}^k |g_i|$, where $|\cdot|$ now denotes the (geodesic) word length in $G = \BS(1,n)$. For a word $w$ on the generators $\{a,t\}$, the notation $\sigma_a(w), \sigma_t(w)$ denotes, respectively, the $a$-exponent and $t$-exponent of $w$.

\subsection{Equivalence of equations}
We describe two equations $W = 1$, $V=1$ as \emph{equivalent} if there is an automorphism $\phi\in \Aut(F*G)$ such that $\phi$ is the identity on $G$ and $\phi(W) = V$. It is a well known consequence of the classification of compact surfaces that any quadratic equation over $G$ is equivalent, via an automorphism $\phi$ computable in time $O(|W|^2)$, to an equation in exactly one of the following three standard forms (see, for instance, \cite{Grigorchuk-Kurchanov:1992} or \cite{Kharlampovich-Lysenok-Myasnikov-Touikan:2010}):

\begin{align}
W(\bar{z}) &= \prod_{j=1}^k z_j^{-1} c_j z_j=1 &k\ge 1,\label{eq:spherical}\\
W(\bar{x},\bar{y},
\bar{z}) &= \prod_{i=1}^g[x_i,y_i]\prod_{j=1}^k z_j^{-1} c_j z_j=1 &g\geq 1, k\geq 0, \label{eq:orientable}\\
W(\bar{x},\bar{z}) &= \prod_{i = 1}^g x_i^2\prod_{j=1}^k z_j^{-1} c_j z_j=1 &g \geq 1, k \geq 0.\label{eq:nonorientable}
\end{align}
The number $g$ is the \emph{genus} of the equation, and both $g$ and $k$ (the number of constants) are invariants. The standard forms are called, respectively, 
\emph{spherical} (or \emph{orientable of genus $0$}), \emph{orientable of genus $g$}, and \emph{non-orientable of genus $g$}. An equation is called spherical, orientable of genus $g
\geq 1$ or nonorientable according to its standard form. Analogously to the case of compact surfaces, if $W=1$ is an equation whose variables each occur as both $z_i$ and $z_i^{-1}$, then it is orientable (possibly of genus 0, i.e. spherical); otherwise, it is nonorientable. Assuming that $G$ is finitely generated, it can be shown that the standard form $\phi(W)=1$ is not larger than the original word length $|W|$ (more details may be found in \cite{Grigorchuk-Kurchanov:1992}). Thus, we may consider only equations in these standard forms.

\subsection{Basic facts and lemmas for $\BS(1,n)$}
\subsubsection{The semidirect product representation}
Let $\varphi:\MZ\to\Aut(\MZ[\tfrac{1}{n}])$ be the homomorphism 
which maps $1$ to $x\mapsto \frac{x}{n}$.
We make use of the well-known fact that $$\MZ[\tfrac{1}{n}] \rtimes_\varphi \MZ\cong \BS(1,n)$$ 
via the isomorphism $\psi$ defined by \begin{equation}\label{metabelian:normalform}
    (yn^{-\epsilon},b) \mapsto t^\epsilon a^y t^{-\epsilon} t^b.\end{equation} Note that $$\psi(\MZ[\tfrac{1}{n}]) = \{t^{\epsilon}a^y t^{\epsilon} \ | \ \epsilon, y \in \MZ\} = \langle\! \langle a \rangle\! \rangle,$$ and the second component of $\psi^{-1}(g)$ is equal to $\sigma_t(g)$. Where convenient, an element of $\BS(1,n)$ may be written as the corresponding element of $\mathbb Z[\tfrac{1}{n}] \rtimes \mathbb Z$ without further mention.
    
    The following lemma gives the normal form corresponding to \eqref{metabelian:normalform}, as well as some useful bounds. Lemmas \ref{le:semidirect-bound} and \ref{le:log-length} provide useful bounds on, respectively, the semidirect product representation in terms of $|W|$, and the length of a word $W$ corresponding to a given element of $\MZ[\tfrac{1}{n}]\rtimes \MZ$.  

\begin{lemma}\label{le:standard-word}
For every element $g\in \BS(1,n)$, there exist $\epsilon \in \MZ_{\geq 0}$ and $y,z\in\MZ$ such that
\begin{equation}\label{eq:standard-word}
    g=t^\epsilon a^y t^{-\epsilon}t^z.
\end{equation}
Moreover, if $W$ is a word on the generators $\{a,t\}$ representing $g$, then 
\begin{align*}
  |y| < |W||n|^{|W|} \ \ \ \mbox{and} \ \ \  \epsilon, |z| \leq |W|.
\end{align*}
\end{lemma}
\begin{proof}
   It is clear that an arbitrary word may be put into the normal form \eqref{eq:standard-word} by performing successive transformations of the form $t^{-1}a \rightarrow a^nt^{-1}$ and $at \rightarrow ta^n$. In particular, each time a positive $t$-letter is moved all the way to the left (or a negative $t$-letter to the right), the absolute value of the $a$-exponent is multiplied by a factor less than or equal to $|n|$. This shows that the bound for $y$ holds, and since the number of $t$-letters does not increase, the bounds for $\epsilon,|z|$ are immediate.
\end{proof}

\begin{lemma}\label{le:semidirect-bound}
Let $W$ be a word on the generators $\{a,t\}$, and let $(\alpha,\beta) \in \MZ[\tfrac{1}{n}] \rtimes \MZ$ be such that $W=(\alpha,\beta)$ in $\BS(1,n)$. Then we have
\begin{enumerate}[(a.)]
    \item \label{ineq:alpha}
    $|\alpha| < |n|^{|W|}$
    \item \label{ineq:beta}
    $|\beta|\leq |W|$
    \item \label{make-int}
    $n^{|W|}\alpha \in \MZ$.
\end{enumerate}
\end{lemma}
\begin{proof}
We may write $W$ as $t^{\epsilon_0}a^{\delta_0} t^{\epsilon_1}a^{\delta_1}\cdots t^{\epsilon_k}a^{\delta_k}$, where $\epsilon_i,\delta_i\in \{-1,0,1\}$ and $\sum_{i=0}^k |\epsilon_i|+|\delta_i| = |W|$. Therefore $$(\alpha,\beta) = \prod_{i=0}^k (0,\epsilon_i)(\delta_i,0) = \prod_{i=0}^k (\delta_in^{-\epsilon_i},\epsilon_i) = \left (\sum_{i=0}^k \delta_i n^{-\sum_{j=0}^{i-1}\epsilon_j}, \sum_{i=0}^k \epsilon_i \right),$$
whence \eqref{ineq:beta} and \eqref{make-int} follow immediately. From the right hand side, we see that
$$|\alpha| \leq \sum_{i=0}^k |\delta_i||n|^{\sum_{j=0}^{i-1}|\epsilon_j|} \leq 1+|n|+\cdots +|n|^{k-1} < |n|^{|W|}, $$
proving \eqref{ineq:alpha}.
\end{proof}

\begin{lemma}\label{le:log-length}
Assume that $|n|>1$. Let $g\in \BS(1,n)$, and suppose that $g = (\alpha n^{u},y)$ in the semidirect product notation, with $\alpha, y, u 
\in 
\MZ$. Then $g$ may be written as a word on the generators $\{a,t\}$ of length less than $$2|n|(1+\log_{|n|}|\alpha|) + 2|u| + |y| .$$
\end{lemma}

\begin{proof}
Assume that $\alpha  
\neq 0$  (otherwise the statement is obvious), and let $L = \log_{|n|} |\alpha|$. Then $g$ may be written as
\begin{align*}
w(a,t) = t^{-u}\left(\prod_{i=0}^{\lfloor L \rfloor} t^{-i}a^{\epsilon_i d_i}t^{i}\right )t^{u}t^y &= t^{-u}(a^{\epsilon_0 d_0} t^{-1}a^{\epsilon_1 d_1} t^{-1}
\cdots t^{-1}a^{\epsilon_{\lfloor L \rfloor} d_{\lfloor L \rfloor}}t^{\lfloor L \rfloor})t^{u}t^y\\
&= t^{-u} a^{\epsilon_0 d_0} t^{-1}a^{\epsilon_1 d_1} t^{-1}
\cdots t^{-1}a^{\epsilon_{\lfloor L \rfloor} d_{\lfloor L \rfloor}}t^{\lfloor L \rfloor+u}t^y, 
\end{align*}

where $d_{\lfloor L\rfloor}\cdots d_1d_0$ is the base-$|n|$ expansion of $|\alpha|$ and $\epsilon_i 
\in \{\pm 1\}$. Hence, we have 

\begin{align*}
  |w| \leq 2|u| + 2L + d_0+\cdots+d_{\lfloor L \rfloor} + |y| &< 2L + |n|(L+1) + 2|u| + |y|\\ &
  <2|n|(L+1)+2|u|+|y| \\
  &=2|n|(1+\log_{|n|}|\alpha|) + 2|u| + |y| .  
\end{align*}
\end{proof}

It follows from Lemmas \ref{le:semidirect-bound} and \ref{le:log-length} (and their proofs) that we may efficiently (i.e., in time and space polynomial in $|W|$) transform constants encoded as words on $\{a,t\}$ into elements of $\MZ[\frac{1}{n}]\rtimes \MZ$ (and vice-versa). Hence, from the point of view of $\NP$-completeness, it makes no difference which of the two formats we use for encoding the input. However, solution size will always be considered in terms of the length of words on $\{a,t\}$.

\subsubsection{Commutator and verbal width}

Recall that the \emph{commutator width} of a group $G$
is the minimal number $l\in\MN$ such that every 
element $g$ of the derived subgroup $G'$ can be expressed as a product of at most $l$ commutators
in $G$. Similarly, if $W(z_1,\ldots,z_k)$ is a word on the letters $z_i$, then the \emph{verbal width} (or just \emph{width}) of the verbal subgroup $S = \langle \{W(g_1,\ldots,g_k)| g_i\in G\} \rangle $ is defined as the minimum $l$ such that any $g\in S$ may be expressed as a product of at most $l$ words of the form $W(g_1,\ldots,g_k)$. 

\begin{lemma}
\label{le:derived-sbgp-conditions}
A group word $w=w(a,t)$ represents an element of the derived subgroup $\BS(1,n)'$
if and only if 
\begin{equation}
\label{eq:derived-sbgp-conditions}
\sigma_t(w)=0
\ \ \mbox{ and }\ \ 
\sigma_a(w) \mbox{ is a multiple of } n-1.
\end{equation}
\end{lemma}

\begin{proof} $w$ belongs to $\BS(1,n)'$
if and only if it is trivial in the abelianization $\BS(1,n)^{\text{ab}} = \gpr{a,t}{[a,t]=1,\,a^{n-1}=1} \cong \MZ_{n-1} \times \MZ$, if and only if conditions \eqref{eq:derived-sbgp-conditions}
are satisfied.
\end{proof}

\begin{lemma}\label{le:comwidth}
$\BS(1,n)$ has commutator width $1$. 
\end{lemma}

\begin{proof}
Suppose that $w 
\in 
BS(1,n)'$. Then it follows from Lemmas \ref{le:standard-word}
and \ref{le:derived-sbgp-conditions} that $$w = t^p a^{k(n-1)} t^{-p}$$ for some $p\in \MZ_{
\ge 0
}$ and $k\in\MZ$, and it is easily verified that the right hand side is equal to $[t,a^{-k} t^{-p}]$ (in fact, this result can be proved more generally for any abelian-by-cyclic group).
\end{proof}

Consider the verbal subgroup 
$S=\gp{x^2 \mid x\in\BS(1,n)}$.
As a verbal subgroup, $S$ is normal in $\BS(1,n)$, and it is easy to check that it contains the derived subgroup (which is true in any group).
In particular, the quotient group
$\BS(1,n)/S$
is isomorphic to the finite abelian group
$\gpr{a,t}{a^2,\ a^{n-1},\ t^2,\ [a,t]}$
and

\begin{equation}\label{eq:verbwidth}
w=w(a,t)\in S  
\ \ \Leftrightarrow\ \ 
\gcd(2,n-1)\mid \sigma_a(w) \ \ \wedge\ \ 2\mid \sigma_t(w).  
\end{equation}

\begin{lemma}
\label{le:square-verbal-width}
The verbal width of $S$ is $2$.
\end{lemma}

\begin{proof}
Consider any $s=(\alpha,2b) \in S$.
We may assume that $\alpha\in\MZ$ (otherwise
conjugate $s$ by an appropriate power of $t$).
For $p,q\in\MZ$ define
$$
x = (p,0) 
\ \ \mbox{ and }\ \ 
y = 
\begin{cases}
(qn^b, b) \mbox{ if }b \geq 0, \\
(q, b) \mbox{ if }b < 0,
\end{cases}
$$
so that
$$
x^2y^2 = (2p+(1+n^{|b|})q,2b). 
$$
If $n$ is even, then $\gcd(2,1+n^{|b|}) = 1$ 
and we can choose $p$ and $q$ such that 
$x^2y^2 = (\alpha,2b) = s$.
If $n$ is odd, then $\gcd(2,1+n^{|b|}) = 2$. By \eqref{eq:verbwidth} $\sigma_a(s)$ is even, so $\alpha$ is even and, as above, we may choose $p,q$ in order to obtain $x^2y^2=s$.
\end{proof}

\subsection{The 3-partition problem and the partition problem}
For a multiset $S=\{a_1,\ldots,a_{3k}\}$ of $3k$ integers, let 
$$
L_S = \frac{1}{k}\sum_{i=1}^{3k}a_i.
$$ 
The \emph{$3$-partition problem} 
(we occasionally use the abbreviation 3PART) is the problem of deciding whether a given $S=\{a_1,\ldots,a_{3k}\}$ with $L_S/4 < a_i < L_S/2$ can be partitioned into $k$ triples, each of which sums to $L_S$. This problem is known to be \emph{strongly} $\NP$-complete, which means that it remains $\NP$-complete even when the integers in $S$ 
are bounded above by a polynomial in the input size or, equivalently, if the input is represented in unary. Note that because of the restriction $L_S/4 < a_i < L_S/2$, we may assume that $S$ contains only positive integers.

The \emph{partition problem}, (or $\PART$) is a related problem to decide if a given multiset $S$
of positive integers 
can be partitioned into two subsets $S_1$ and $S_2$ 
such that the sum of the numbers in $S_1$ equals
the sum of the numbers in $S_2$.
The problem is $\NP$-complete if the numbers are given in 
binary, and polynomial time decidable if they are given in unary (i.e. \emph{weakly} $\NP$-complete). For a thorough treatment of these problems, the interested reader is referred to \cite{Garey-Johnson:1979}.

We note that the strong $\NP$-hardness of $\TPART$ is essential to prove $\NP$-hardness of $\DP(\BS(1,n))$ for $|n|>1$, because of the fact that instances of $\DP(\BS(1,n))$ are given as words on the generators, i.e. encoded in unary. On the other hand, the fact that $\PART$ is decidable in polynomial time for unary input will be used in the proof that $\DP(\BS(1,-1))\in \P$. 

\subsection{B\'{e}zout coefficients bound}

In Section \ref{se:metabelian-equations} we use the following lemma.

\begin{lemma}\label{le:bezout}
For any integers $\beta_1,\ldots,\beta_k$ (such that at least one is not zero), there exist \emph{B\'{e}zout coefficients} $s_i$ such that 
\begin{equation}\label{eq:bezout}
\sum_{i=1}^k s_i\beta_i = \gcd(\beta_1,\ldots,\beta_k),
\end{equation}
satisfying $|s_i| < |\beta_1| + \cdots + |\beta_k|.$
\end{lemma}

\begin{proof}
We may assume that $\gcd (\beta_1,\ldots,\beta_k) = 1$, $\beta_1 > \beta_2 > \ldots > \beta_k\ge 1$ and $k > 1$.
Notice that
$s_1,\ldots,s_{k-1},s_k \in \MZ$ satisfy \eqref{eq:bezout}
if and only if 
$s_1,\ldots,s_{k-1}+ c\beta_k,s_k-c\beta_{k-1}$ satisfy 
\eqref{eq:bezout} for any $c\in\MZ$.
Hence, we may assume that $|s_k|<\beta_{k-1}$.
Similarly, we may assume that $|s_i|<\beta_{i-1}$ for $i=2,\ldots,k$.
Now we can directly check that
\begin{align*}
|s_1| =\frac{1}{\beta_1} \abs{1 - \sum_{i=2}^k s_i\beta_i} &\leq 
\frac{1}{\beta_1}\left (1+ \sum_{i=2}^k |s_i|\beta_i \right )
< \frac{1}{\beta_1} + \sum_{i=2}^k |s_i| < \beta_1 + \cdots + \beta_{k},
\end{align*}

completing the proof.
\end{proof}

\section{A polynomial-time algorithm for $\DP(\BS(1,-1))$}\label{se:polytime}
Here, we briefly cover the special case of $\DP(\BS(1,-1))=\gpr{a,t}{t^{-1}at=a^{-1}} \simeq \MZ\ltimes \MZ$, proving that it is decidable in polynomial time.

Notice that each element of $\BS(1,-1)$ can be expressed as 
$a^m t^k$ in a unique way and the following formula holds:
\begin{equation}\label{eq:odd-even}
(a^{y}t^{x})^{-1}
\cdot a^m t^k \cdot
(a^{y}t^{x}) =
\begin{cases}
a^{-m}t^k & \mbox{ if $x$ is odd and $k$ is even,}\\
a^{2y-m}t^k & \mbox{ if $x$ is odd and $k$ is odd,}\\
a^{m}t^k & \mbox{ if $x$ is even and $k$ is even,}\\
a^{m-2y}t^k & \mbox{ if $x$ is even and $k$ is odd.}
\end{cases}    
\end{equation}

This formula immediately implies the following proposition.

\begin{proposition}\label{pr:kbg-spherical}
A spherical equation $W=\prod z_i^{-1} a^{m_i}t^{k_i}z_i=1$
has a solution if and only if $\sum k_i=0$ and either of the following conditions holds.
\begin{itemize}
\item
$\{|m_1|,\dots,|m_k|\}$ is a positive instance of the partition problem.
\item
Some $k_i$ is odd and $\sum m_i$ is even.
\end{itemize}
\end{proposition}

\begin{proposition}\label{pr:kbg-orient}
An orientable equation $W=\prod [x_j,y_j]\prod z_i^{-1}  a^{m_i}t^{k_i}z_i=1$ 
has a solution if and only if $\sum k_i=0$ and $\sum m_i$ is even.
\end{proposition}

\begin{proof}
The statement immediately follows from
$$
\BS(1,-1)' = \Set{[x,y]}{x,y\in\BS(1,-1)} = \Set{a^{2i}}{i\in \MZ}.
$$
\end{proof}

\begin{proposition}\label{pr:kbg-nonorient}
A non-orientable equation $W=x^2\prod z_i^{-1} a^{m_i}t^{k_i} z_i=1$ 
has a solution if and only if any of the following conditions holds
\begin{itemize}
\item
$\sum k_i\equiv_4 2$ and
$\{|m_1|,\dots,|m_k|\}$ is a positive instance of the partition problem.
\item
$\sum k_i\equiv_4 2$ and  
some $k_i$ is odd and
$\sum m_i$ is even.
\item
$\sum k_i\equiv_4 0$ and $\sum m_i$ is even.
\end{itemize}
\end{proposition}

\begin{proof}
The result follows immediately from \eqref{eq:odd-even} and from
$$
\Set{x^2}{x\in\BS(1,-1)} = 
\Set{t^{4j+2}}{j\in \MZ}
\cup 
\Set{a^{2i}t^{4j}}{i,j\in \MZ}.
$$
\end{proof}

\begin{proposition}\label{pr:kbg-non-genus}
A non-orientable equation $W=x_1^2\cdots x_n^2\prod z_i^{-1} a^{m_i}t^{k_i} z_i=1$,
where $n\ge 2$,
has a solution if and only if
$\sum k_i$ and $\sum m_i$ are both even.
\end{proposition}

\begin{proof}
The statement is an immediate consequence of the following formula:
$$
\gp{\Set{w^2}{w\in \BS(1,-1)}}=\Set{x^2y^2}{x,y\in\BS(1,-1)} = 
\Set{a^{2i}t^{2j}}{i,j\in \MZ}.
$$
\end{proof}

\begin{theorem}
The Diophantine problem for quadratic equations over $\BS(1,-1)$
is decidable in polynomial time. Moreover, if a quadratic equation over $\BS(1,-1)$ has a solution, then it has a solution whose size is $O(|W|)$.
\end{theorem}

\begin{proof}
It requires linear time to compute $\sum k_i$ and $\sum m_i$.
Furthermore, the partition problem for 
numbers $\{m_1,\dots,m_k\}$ given in unary can be solved in polynomial time.
Hence, Propositions \ref{pr:kbg-spherical}, \ref{pr:kbg-orient}, \ref{pr:kbg-nonorient} and \ref{pr:kbg-non-genus} give an efficient solution to the problem, and the solution bound is easily deduced from each of their proofs. 
\end{proof}
Having established that Theorem \ref{th:solution-bound} holds for $\BS(1,\pm 1)$, we henceforth assume that $|n| > 1$.

\section{Complexity lower bounds for quadratic equations over $\BS(1,n)$}
\label{se:np-hard}

In this section, we prove $\NP$-hardness for the quadratic Diophantine problem over $\BS(m,n)$ for all $m,n$ such that $|m| \neq |n|$, thereby proving Theorem \ref{th:gen-np-hard} (in a forthcoming paper, the authors prove the same result for the unimodular Baumslag-Solitar groups $\BS(n,\pm n)$). Additionally, we prove $\NP$-hardness for the special case of nonorientable equations of genus 1 over $\BS(1,n)$ with $|n|>1$.
\subsection{Algebraic equations with exponents}\label{se:exp_eq}
An \emph{algebraic equation with exponents} is an equation of the form
\begin{equation}\label{eq:alg_exp}
q_1\alpha^{x_1}+\cdots +q_k\alpha^{x_k} = q,
\end{equation}
where $\alpha\in\MR$ is algebraic,
$q_i,q,\in \MQ(\alpha)$ and unknowns $x_i$, and for which integer solutions 
are required.
It is proved in \cite{Mandel-Ushakov:2022-b} that 
the Diophantine problem for equations \eqref{eq:alg_exp},
i.e., the problem to decide if \eqref{eq:alg_exp} has a solution or not,
is $\NP$-complete for any fixed $\alpha \not \in \{-1,0,1\}$. 
We require the following results on algebraic equations with exponents 
(with $\alpha,q_1,\dots,q_k,q\in\MQ$) in order to establish an upper bound
on solution size (to prove Theorem \ref{th:solution-bound}) and to enable the reduction $\TPART \leq \DP$. The proofs of Theorem \ref{th:Semenov-NP} and Propositions \ref{pr:unique_solution_exp_eq} and \ref{pr:3-part} are given in \cite{Mandel-Ushakov:2022-b}; here, we include only the proof of Proposition \ref{pr:3-part}, because it involves a general procedure that is also invoked to prove $\NP$-hardness for the nonorientable case.

\begin{theorem}
\label{th:Semenov-NP}
Let $q_1,\ldots,q_k \in \MZ$, $\alpha\in \MQ$ and suppose that the equation
\begin{equation}\label{eq:alg_exp_int}
q_1\alpha^{x_1}+\cdots +q_k\alpha^{x_k} = 0
\end{equation}
has an integer solution. Then it has a solution $x_1,\ldots,x_k\in\MZ$ satisfying
\begin{equation}\label{eq:Semenov-bound}
0
\ \le\ 
x_1,\ldots,x_k
\ \le\ 
\sum_{i=1}^k \abs{\log_{|\alpha|} (|q_i|+1)}.
\end{equation}
\end{theorem}

\begin{proof}
The statement follows from
\cite[Theorem 3.5]{Mandel-Ushakov:2022-b}
and
\cite[Proposition 3.11]{Mandel-Ushakov:2022-b}
\end{proof}

\begin{proposition}[{{\cite[Proposition 2.3]{Mandel-Ushakov:2022-b}}}]
\label{pr:unique_solution_exp_eq}
Given $s\in \MN$, there exists $c  \in\MN$
such that for any $\alpha\in\MQ\setminus \{-1,0,1\}$ and integers $0\le p_1<p_2<\cdots<p_s$ satisfying $p_{i+1}-p_{i} \geq c$,
the equation
\begin{equation}\label{eq:exp_equality}
\alpha^{x_1}+\cdots+\alpha^{x_s}=
\alpha^{p_1}+\cdots+\alpha^{p_s}
\end{equation}
has (up to a permutation) the unique integer solution $x_i = p_i$. Furthermore, $c$ is $O(\log s)$, and may be computed efficiently in terms of $s$.
\end{proposition}

\begin{proposition}\label{pr:3-part}
Fix $\alpha\in\MQ\setminus\{-1,0,1\}$. Let $S=\{a_1,\ldots,a_{3k}\}$ be an instance of 
$\TPART$, with $L = \frac{1}{k}\sum a_i$ the anticipated sum for subsets, where 
$L/4 < a_i < L/2$ for $i=1,2,\ldots,3k$ (we may assume that $L\in\MN$). Let $c\in \MN$ be the number guaranteed by Proposition \ref{pr:unique_solution_exp_eq}, corresponding to $s = Lk$. Define numbers
\begin{align*}
q_y &= 1+\alpha^c+\alpha^{2c}+\cdots+\alpha^{(y-1)c}
\ \ \ \ \ \ \mbox{ for }\ \  y\in\MN\\
r &= q_L\rb{1+\alpha^{2cL}+\alpha^{4cL}+\cdots+\alpha^{2(k-1)cL}}
\end{align*}
and the equation
\begin{equation}\label{eq:3-part}
q_{a_1}\alpha^{x_1}+\cdots+q_{a_{3k}}\alpha^{x_{3k}} = r.
\end{equation}
Then $S$ is a positive instance of $\TPART$
if and only if \eqref{eq:3-part} has a solution.
Moreover, each solution for \eqref{eq:3-part} satisfies
\begin{equation}\label{eq:3part-red-bound}
0\le x_1,\ldots,x_{3k}\le 2ckL.
\end{equation}
\end{proposition}

\begin{proof}
Suppose that $S$ is a positive instance of $\TPART$. 
Reindexing the $a_i$ and $x_i$ if necessary, we may assume that $\sum_{j=1}^3 a_{3i+j} = L$ for $i = 0,1,\ldots, k-1$. It is now easily checked that
\begin{align*}
x_{3i+1} &= 2icL \\
x_{3i+2} &= c(2iL + a_{3i+1}) \\
x_{3i+3} &= c(2iL + a_{3i+1} + a_{3i+2})
\end{align*} 
for $i = 0,1,\ldots,k-1$
satisfies \eqref{eq:3-part} and \eqref{eq:3part-red-bound}.

For the other direction, suppose that $x_1,\ldots,x_{3k}$ is a solution
of \eqref{eq:3-part}. 
By construction, the left hand side of (\ref{eq:3-part}) is 
a sum of $Lk$ powers of $\alpha$, while 
the right hand side is a sum of $Lk$ distinct powers of $\alpha^c$. In particular, the sum on the right hand side contains blocks of consecutive powers of $\alpha^c$, with gaps between $\alpha^{(2i-1)c(L-1)}$ and $\alpha^{2icL}$, $i = 1,\ldots,k-1$. Proposition  \ref{pr:unique_solution_exp_eq} implies 
that the left hand side consists of the same distinct 
powers of $\alpha$, and the proof follows from a careful comparison of these powers. First, it is clear that we must have $x_{i_1} = 0$ for exactly one $x_{i_1}$,
and
$q_{a_{i_1}}\alpha^{x_{i_1}} = 
1 + \alpha^c + \ldots + \alpha^{c(a_{i_1}-1)}$.
Since $a_{i_1} < L$ by assumption, the right hand side of \eqref{eq:3-part} contains $\alpha^{ca_{i_1}}$ and so we must have $x_{i_2} = ca_{i_1}$ for some (unique) $x_{i_2}$.
Similarly, the highest degree term of 
$q_{a_{i_1}}\alpha^{x_{i_1}} + q_{a_{i_2}}\alpha^{x_{i_2}}$ 
is $\alpha^{c(a_{i_1}+a_{i_2}-1)}$, and since $c(a_{i_1}+a_{i_2}) < cL$, we must have the next consecutive power of $\alpha^c$ on the left hand side. Hence, there must be
$x_{i_3} = c(a_{i_1}+a_{i_2})$. Finally, the highest power
of $q_{a_{i_1}}\alpha^{x_{i_1}} + q_{a_{i_2}}\alpha^{x_{i_2}} + q_{a_{i_3}}\alpha^{x_{i_3}}$ is $\alpha^{c(a_{i_1}+a_{i_2} + a_{i_3}-1)}$, which must be the last term in the first block of consecutive powers of $\alpha^c$ (since for any $a_l$, we have $cL < c(a_{i_1}+a_{i_2}+a_{i_3} + a_l)$). This implies that $c(a_{i_1}+a_{i_2}+a_{i_3}) = cL$, and the next largest $x_i$ is equal to $2cL$. It is clear that this process may be continued to show that $S$ is a positive instance, and that $x_1,
\ldots,x_{3k}$ satisfies \eqref{eq:3part-red-bound}. 
\end{proof}

\subsection{$\NP$-hardness for $\DP(\BS(m,n))$ where $|m|\ne |n|$}
\label{sec:NPhard}

Let us fix $m,n \in \mathbb{Z}\setminus \{0\}$ such that $|m| \neq |n|$, and let $\alpha = \frac{m}{n}$. Let $S=\{a_1,\ldots,a_{3k}\}$
be an instance of $\TPART$ with $L/4 < a_i < L/2$,  where $L$ is the anticipated sum. 
Let $c\in\MN$ be defined as in Proposition \ref{pr:3-part}.
Construct the equation \eqref{eq:3-part} corresponding to $S$, with integer coefficients $q_{a_1},\dots,q_{a_{3k}}$ and $r$
as defined Proposition \ref{pr:3-part}.
Multiply the coefficients by $n^{4ckL}$ to get
$$
b = rn^{4ckL}
\ \ \mbox{ and }\ \ 
b_i = q_{a_i}n^{4ckL}
\ \ \mbox{ for } i=1,\dots,3k.
$$
Finally, consider the spherical equation 
\begin{equation}\label{eq:spherical-hard}
z_1^{-1} a^{b_1} z_1 \cdots z_{3k}^{-1} a^{b_{3k}} z_{3k} = a^{b}
\end{equation}
over the group $\BS(m,n)$.

\begin{proposition}\label{pr:NP-hardness-DP}
$S$ is a positive instance of $\TPART$
if and only if 
\eqref{eq:spherical-hard}
has a solution.
\end{proposition}

\begin{proof}
If $S$ is a positive instance of $\TPART$, then by Proposition \ref{pr:3-part}, \eqref{eq:3-part} has a solution $x_1,\ldots,x_{3k}\in\MZ$
satisfying $0\le x_1,\ldots,x_{3k} \le 2ckL$. 
By construction, we have $ \frac{b_i}{n^{x_i}} \in\MZ$ for all of the $x_i$.
Hence, in $\BS(m,n)$ we have
$t^{-x_i}a^{b_i}t^{x_i}=a^{b_i \rb{\frac{m}{n}}^{x_i}}$
and
$$
\prod t^{-x_i}a^{b_i}t^{x_i}=
\prod a^{b_i \rb{\frac{m}{n}}^{x_i}}=
a^{\sum b_i \rb{\frac{m}{n}}^{x_i}}=a^b.
$$
Thus, $z_i=t^{x_i}$ is a solution of \eqref{eq:spherical-hard}.

Conversely, suppose that \eqref{eq:spherical-hard} has a solution
$z_1,\ldots,z_{3k}\in\BS(m,n)$.
Since the constants do not involve $t$-letters, 
we may assume that each $z_i = t^{x_i}$ for some $x_i\in\MZ$.  
Applying Britton's lemma, it is easily seen that the algebraic equation with exponents
$$
b_1 (\tfrac{m}{n})^{x_1}+\cdots+b_{3k} (\tfrac{m}{n})^{x_{3k}} = b
$$
must hold. Therefore, by Proposition \ref{pr:3-part},
$S$ is a positive instance of $\TPART$.
\end{proof}

\begin{corollary}\label{co:spherical-NP-hard}
The Diophantine problem for $\BS(m,n)$ is strongly $\NP$-hard whenever $|m| \neq |n|$ (i.e. Theorem \ref{th:gen-np-hard} holds).
\end{corollary}


\subsection{$\NP$-hardness for nonorientable equations of genus 1}

In this section we fix $n
\in \MZ$ such that $|n|>1$. Consider an instance of the $3$-partition problem 
$S=\{a_1,\ldots,a_{3k}\}$
with $L/2 < a_i < L/4$, where $L$ is the anticipated sum. For $s\in \MN$, define
$$A(s) = \sum_{i=1}^s n^{ikL} \mbox{ and } A^* = \sum_{i=0}^{k-1} n^{ikL(L+1)}A(L),$$
and let $M = k^2L(L+1)$. Finally, consider the following nonorientable equation of genus $g=1$ over the group $\BS(1,n)$:
\begin{equation}\label{eq:gen1reduction}
W = x^2
(y^{-1} t^{2M} y)
a^{-A^*}
\prod_{i=1}^{3k} z_i^{-1}a^{A(a_i)} z_i = 1.
\end{equation}

\begin{proposition}\label{pr:non-orient1}
The equation \eqref{eq:gen1reduction} 
has a solution if and only if $S$ is a positive instance 
of the $3$-partition problem.
\end{proposition}

\begin{proof}
``$\Leftarrow$''
Suppose that $S$ is a positive instance.
Reindexing $a_1,\ldots,a_{3k}$ if necessary, we may assume that 
$a_{3j+1}+a_{3j+2}+a_{3(j+1)} = L$ for each $j = 0,\ldots , k-1$. 
Let us define $\alpha_1,\ldots,\alpha_{3k} \in\MN$ as follows:
\begin{align*}
\alpha_1 &= 0       &\alpha_4&= (L+1)        &\alpha_{3i+1}&= i(L+1) \\
\alpha_2 &= a_1     &\alpha_5&= (L+1)+a_4    &\alpha_{3i+2}&= i(L+1)+a_{3i+1} \\
\alpha_3 &= a_1+a_2 &\alpha_6&= (L+1)+a_4+a_5&\alpha_{3i+3}&= i(L+1)+a_{3i+1}+a_{3i+2},
\end{align*}
for $i = 0,1,\ldots,k-1$.
Then $x=t^{-M}$, $y=1$ and $z_i = t^{kL\alpha_i}$ 
is a solution to \eqref{eq:gen1reduction} in the
free group $F(a,t)$.

``$\Rightarrow$''
Suppose that \eqref{eq:gen1reduction} has a solution.
First, we prove the following auxiliary statement.

\begin{lemma}\label{lem:meta-hardness}
Fix $K,M,p_1,\ldots,p_K\in\MN$ and $n\in \MZ$ such that $M$ is even, $|n|>1$, 
$p_1<p_2<\cdots<p_K<M/2$, and $p_{i+1}-p_i \geq K$. Let $0 \leq l \leq K$.
Consider the equation
\begin{equation}\label{eq:sum385}
X(n^M+1) = 
(n^{p_1}+\cdots+n^{p_K}) - \sum_{i=1}^{l} \varepsilon_i n^{x_i}
\end{equation} 
with unknowns $x_i,\varepsilon_i,X$. 
Every integer solution of \eqref{eq:sum385} that satisfies
\begin{equation}\tag{*}
0\le x_1,\ldots,x_l<M 
\ \mbox{ and }\ 
\varepsilon_i \in \{-1,1\}
\end{equation}
 necessarily satisfies the following:
\begin{enumerate}[(i)]
\item 
$X=0$,
\item 
$\varepsilon_1=\cdots=\varepsilon_l=1$,
\item 
$l = K \mbox{ and }\{x_1,\ldots,x_l\} = \{p_1,\ldots,p_K\}$.
\end{enumerate}
\end{lemma}

\begin{proof}
Let $l$ be the minimum number for which there is a solution for (\ref{eq:sum385}) 
satisfying (*). No such solution exists when $l=0$ and
there is an obvious solution $X=0$, $\varepsilon_i = 1$, $x_i = p_i$ when $l = K$.
Hence, we may assume that $l\in \{1,2,\ldots,K\}$.
Below we prove that $l = K$.
 
Let $x_i,\varepsilon_i,X$ be a solution to (\ref{eq:sum385}) 
that satisfies (*). 
If $\varepsilon_in^{x_i}+\varepsilon_j n^{x_j}=0$, 
then eliminating $\varepsilon_in^{x_i}$ and $\varepsilon_j n^{x_j}$
and reindexing terms we decrease $l$ by $2$, contradicting minimality. 
Thus, the same power cannot appear twice with opposite signs.
Similarly, if $x_{i_1} =\cdots = x_{i_n} = s$ for $n$ distinct
indices $i_1,\ldots,i_n$, then by the foregoing argument $\varepsilon_{i_1} =\cdots = \varepsilon_{i_n}$, 
and we may perform the substitution
\begin{align*}
\varepsilon_{i_1} n^{x_{i_1}}+ \cdots + \varepsilon_{i_n} n^{x_{i_n}} &\to \varepsilon_{i_1} n^{s+1} &&\mbox{if }s<M-1,\\
\varepsilon_{i_1} n^{x_{i_1}}+ \cdots + \varepsilon_{i_n} n^{x_{i_n}} &\to -\varepsilon_{i_1}  &&\mbox{if }s=M-1,
\end{align*}
(because $n^M \equiv -1 \pmod{n^M+1}$)
reducing the number of terms by $n-1$.
In both cases the minimality of $l$ is contradicted, so we have shown that that each power may appear at most $n-1$ times, and always with the same sign. Letting $I = \{i\mid \varepsilon_i=1\}$, $J = \{i\mid \varepsilon_i=-1\}
$, this means that
\begin{equation}\label{I-J-bound}
\sum_{i \in I} n^{x_i},\ \sum_{i \in J} n^{x_i}  
\ <\ 
\sum_{\delta=1}^K (n-1)n^{M-\delta},
\end{equation}
and the numbers $\sum_{i\in I} n^{x_i}
$ and $\sum_{i\in J} n^{x_i}$ have base-$n$ representations (where the base is allowed to be negative) whose digits add up 
to $|I|$ and $|J|$, respectively. Observe that \eqref{eq:sum385} implies the congruence
\begin{equation}\label{eq:congruence-IJ}
\sum_{i \in I} n^{x_i} 
\ \equiv\  
n^{p_1} + \cdots +n^{p_K} + \sum_{i \in J} n^{x_i} \ (\mbox{mod } n^M+1).
\end{equation}
The assumptions on $p_1,\ldots,p_K$ ensure that $p_K<M-K$; combining this with \eqref{I-J-bound} we have that both sides of \eqref{eq:congruence-IJ} are between 0 and $n^M+1$. Thus, the congruence is actually an equality, proving that $X=0$. We may now consider the equation
\begin{equation}\label{eq:IJ}
\sum_{i \in I} n^{x_i} 
\ = \ 
n^{p_1} + \cdots +n^{p_K} + \sum_{i \in J} n^{x_i}.
\end{equation}
It is clear that $|J|<K$, since $|I| +|J| \le K$ and $I$ must be nonempty. Considered in base-$n$, the number $n^{p_1} + \cdots +n^{p_K}$ 
is represented as $K$ ones separated by strings of $K-1$ zeros. 
It is easily seen that adding fewer than $K$ powers of $n$ 
to this number cannot reduce the sum of its digits
(at least $(n-1)K$ powers of $n$ would be required to do this).
Hence, the base-$n$ digits of the right hand side must sum to at least $K$, proving that $|I|=K$, $J=\emptyset$, and $l=K$.
The uniqueness of the base-$n$ representation implies that the only solution is 
(up to a permutation of indices) $x_i = p_i$, completing the proof.
\end{proof}

To complete the proof of Proposition \ref{pr:non-orient1},
let $x,y,z_i$ be a solution to \eqref{eq:gen1reduction}. 
It is not hard to see that the following assumptions can be made:
\begin{itemize}
\item
$x=(v_1,-M)$ for some $v_1\in\MZ[\tfrac{1}{n}]$;
\item
$y=(v_2,0)$ for some $v_2\in\MZ[\tfrac{1}{n}]$;
\item
$z_i=t^{u_i}=(0,u_i)$ for each $i=1,\ldots,3k$.
\end{itemize}
Hence, we have
\begin{align*}
x^2y^{-1}t^{2M}y &= (v_1(1+n^M),-2M)(v_2(n^{-2M}-1),2M)\\
&=(v_1(1+n^M)+v_2(1-n^{2M}),0) = ((v_1+v_2(1-n^M))(1+n^M)),0).
\end{align*}
Letting $X = v_1+v_2(1-n^M)$, we thus obtain the equation
\begin{align*}
X(1+n^M) &= A^* - \sum_{i=1}^{3k}n^{u_i}A(a_i)\\
&= \sum_{i=0}^{k-1}n^{ikL(L+1)} 
(n^{kL}+n^{2kL}+\cdots+n^{LkL}) - \sum_{i=1}^{3k}n^{u_i}\rb{n^{kL}+n^{2kL}+\cdots+ n^{a_ikL}}
\end{align*}
from a direct computation in the first component of \eqref{eq:gen1reduction}. Multiplying both sides by a nonnegative power of $n$ if necessary, we may assume that $X\in \MZ$, and it is clear that the assumptions of Lemma \ref{lem:meta-hardness} are satisfied. Thus, an application of Lemma \ref{lem:meta-hardness} shows that the sums $A^*$ and $\sum_{i=1}^{3k}n^{u_i}A(a_i)$ both contain the same $Lk$ powers of $n$. Now, the same procedure used in the proof of Proposition \ref{pr:3-part} (i.e. comparing powers and solving for $u_i$) shows that $S$ is a positive instance, completing the proof.
\end{proof}

\begin{corollary}
    The Diophantine problem for nonorientable equations of genus 1 over $\BS(1,n)$ is strongly $\NP$-hard whenever $|n|>1$.
\end{corollary}

\section{Complexity upper bounds for quadratic equations over $\BS(1,n)$}

\label{se:metabelian-equations}

In this section we classify the computational hardness of the Diophantine problem for several classes of quadratic equations over $\BS(1,n)$. We prove that $\DP(\BS(1,n))$ is 
decidable in linear time for orientable equations of genus $g\ge 1$
and nonorientable equations of genus $g\ge 2$, and that it is $\NP$-complete when restricted to either spherical equations or nonorientable equations of genus 1. Combined with Theorem \ref{th:gen-np-hard}, these results prove Theorem \ref{th:meta-general-complete}. We also prove Theorem \ref{th:solution-bound} (separately, for each of the aforementioned cases). Note that the $O(|W|^3)$ solution bound can be improved in some cases (for instance, we are guaranteed a solution that is $O(|W|)$ when the problem is restricted to nonorientable equations of genus greater than 2); details are provided in each case of the proof of Theorem \ref{th:solution-bound}. The following lemma is used several times in this section.

\begin{lemma}\label{le:easy-fact}
Let $u = \gcd(v_1,\ldots,v_k)$, where $v_1,\ldots,v_k \in \MN$. Then   
\begin{equation}\label{eq:easy-fact}
n^{u}-1 = \gcd(n^{v_1}-1,\ldots,n^{v_k}-1).   
\end{equation}
\end{lemma}

\begin{proof}
For any $a,b\in \MN$ with $b\ge a>0$, we have  $n^{b-a}-1 = (n^{b}-1) - n^{b-a}(n^{a}-1)$, implying that
$$
\gcd(n^{a}-1,n^{b}-1) = \gcd(n^{a}-1,n^{b-a}-1).
$$
More generally, if $r=b-aq$, then
$$
\gcd(n^{a}-1,n^{b}-1) = \gcd(n^{a}-1,n^{r}-1),
$$
so one may run the Euclidean algorithm on the exponents to show that
$$
\gcd(n^{a}-1,n^{b}-1) = n^{\gcd(a,b)}-1.
$$
This proves the statement for $k=2$, and the general result follows easily by induction.
\end{proof}

\subsection{Spherical equations in $\BS(1,n)$}

\begin{proposition}\label{pr:spherical-solution}
Let $c_i=(\alpha_i,\beta_i)\in \MZ[\tfrac{1}{n}] \rtimes \MZ \cong \BS(1,n)$ for $i = 1,\ldots,k$, and let $\beta = \gcd (|\beta_1|,\ldots,|\beta_k|)$. Then the equation
\begin{equation}\label{eq:le-spherical}
\prod_{i=1}^k z_i^{-1}c_iz_i = 1
\end{equation}
has a solution if and only if the following conditions hold:
\begin{enumerate}[(i.)]
    \item \label{beta-cond}
    $\sum_{i=1}^k\beta_i=0$
    \item \label{alpha-cond}
    there exist $x_1,\ldots,x_k\in \MZ$ such that
    $\alpha_1n^{x_1}+\cdots + \alpha_k n^{x_k} \equiv 0 \ (\mbox{mod }n^{\beta}-1)$.
\end{enumerate}
Moreover, if \eqref{eq:le-spherical} has a solution, then it has a solution of size $O(|W|^3)$.
\end{proposition}

\begin{proof} 
Suppose $z_i = (v_i,y_i)$ is a solution to \eqref{eq:le-spherical}. Conjugating $c_i$ by $z_i$ we obtain
$$
(v_i,y_i)^{-1}(\alpha_i,\beta_i)(v_i,y_i)=
(n^{y_i}(\alpha_i+v_i(n^{-\beta_i}-1)),\beta_i),
$$
and so
\begin{equation}\label{eq:alpha-beta-star}
    (0,0) = \prod_{i=1}^k (v_i,y_i)^{-1}(\alpha_i,\beta_i)(v_i,y_i)=
\rb{\sum_{i=1}^k n^{y_i-\sum_{j=1}^{i-1}\beta_j}
(\alpha_i+v_i(n^{-\beta_i}-1)),\sum_{i=1}^k \beta_i},
\end{equation}
which immediately implies the necessity of condition \eqref{beta-cond}. Set
$$
v_i' = \begin{cases}
    n^{-\beta_i}v_i 
    \mbox{ if }\beta_i > 0 \\
    -v_i 
    \mbox{ otherwise}
\end{cases}
$$
so that $v_i(n^{-\beta_i}-1) = -v_i'(n^{|\beta_i|}-1)$,
and let $L$ be a nonnegative integer sufficiently large that $n^{L+y_i-\sum_{j=1}^{i-1}\beta_j}v_i' \in \MZ$ for $i = 1,\ldots,k$. Letting $x_i = L + y_i - \sum_{j=1}^{i-1}\beta_j$, we obtain from \eqref{eq:alpha-beta-star}
\begin{equation}\label{eq:vee-prime}
\alpha_1n^{x_1}+\cdots + \alpha_k n^{x_k} = \sum_{i=1}^k n^{x_i}v_i'(n^{|\beta_i|}-1),
\end{equation}
where the right hand side is a sum of integer multiples of $n^{|\beta_i|}-1$. An application of Lemma \ref{le:easy-fact} now proves that \eqref{alpha-cond} holds.

Conversely, suppose that \eqref{beta-cond} and \eqref{alpha-cond} hold, so there are integers $x_1,
\ldots,x_k$ such that 
\begin{equation}\label{eq:beta-cong}
\alpha_1 n^{x_1}+
\cdots + \alpha_k n^{x_k} \equiv 0 \ (\mbox{mod }n^{\beta}-1).
\end{equation}
We consider three cases.
\begin{case}\label{case1} Suppose that $\beta_1 = \cdots = \beta_k = 0$ (implying that $\beta = 0$), so that the congruence \eqref{eq:beta-cong} becomes an equation of the form \eqref{eq:alg_exp_int}. It is clear that $z_i = t^{x_i}$ is a solution to \eqref{eq:le-spherical} in this case, and (multiplying the $\alpha_i$ by $n^{|W|}$ to ensure integer coefficients) Theorem \ref{th:Semenov-NP} guarantees a solution with 
\begin{align*}
0
\ &\le\ 
x_1,\ldots,x_k
\ \le\ 
\sum_{i=1}^k \log_{|n|} (|\alpha_i||n|^{|W|}+1) \leq \sum_{i=1}^k \log_{|n|}(|n|^{2|W|} + 1) \\
&< k(2|W|+1) \leq 2|W|(|W|+1).    
\end{align*}
Therefore, the size $\sum_{i=1}^k|z_i|$ of the solution is $O(|W|^3)$.
\end{case}
\begin{case}
Assume that at least one of the $\beta_i$ is nonzero, so that $\beta \neq 0$ and
$$\alpha_1 n^{x_1}+
\cdots + \alpha_k n^{x_k} = h(n^{\beta}-1)$$
for some $h\in \MZ$. Since $n^{\beta}\equiv 1 \mod (n^{\beta}-1)$, we may assume that $0\leq x_i<\beta$, and so
$$|h| \leq \frac{n^{\beta}}{n^{\beta}-1}\sum_{i=1}^k |\alpha_i| \leq 2\sum_{i=1}^k |n|^{|c_i|} \leq 2\prod_{i=1}^k |n|^{|c_i|+1}<2|n|^{|W|}, $$
where the second inequality follows from Lemma \ref{le:semidirect-bound}.
By Lemmas \ref{le:bezout} and \ref{le:easy-fact}, there are $s_1,\ldots,s_k\in \MZ$ such that 
\begin{equation*}
\alpha_1 n^{x_1}+
\cdots + \alpha_k n^{x_k} = h\sum_{i=1}^k s_i (n^{|\beta_i|}-1)
\end{equation*}
with $$|s_i|< \sum_{i=1}^k (|n|^{|\beta_i|}+1)< k+\prod_{i=1}^k |n|^{|\beta_i|+1} \leq |W|+  |n|^{|W|} < 2n^{|W|}$$ (using Lemma \ref{le:semidirect-bound} again). Setting 
\begin{align*}
v_i &= \begin{cases}
    n^{\beta_i-x_i}hs_i 
    \ \mbox{   if }\beta_i > 0 \\
    -n^{-x_i}hs_i \ \mbox{   otherwise}
    \end{cases} \\
y_i &= x_i + \sum_{j=1}^{i-1}\beta_j
\end{align*}
we see that $z_i = (v_i,y_i)$ is a solution to \eqref{eq:alpha-beta-star} (hence, to \eqref{eq:le-spherical}). By Lemma \ref{le:log-length} (using the bounds on $|s_i|$ and $|h|$ established above) we have that $(v_i,y_i)$ may be written as a word on $\{a,t\}$ of length less than
\begin{align*}
2|n|(1+\log_{|n|}|hs_i|)+2x_i+2|\beta_i|+|y_i|&< 2|n|(3+2|W|)+3\beta+2\sum_{j=1}^{i}|\beta_j|\\ &\leq 2|n|(3+2|W|)+5|W|,
\end{align*}
which is $O(|W|)$. Hence, the entire solution has size $O(k|W|) = O(|W|^2)$.
\end{case}
\end{proof}

\begin{corollary}\label{th:meta-spherical-complete}
The Diophantine problem for spherical equations
\begin{equation*}
W = \prod_{i=1}^k z_i^{-1}c_iz_i = 1
\end{equation*}
over $\BS(1,n)$ is $\NP$-complete. Moreover, the class of spherical equations satisfies Theorem \ref{th:solution-bound}.
\end{corollary}

\begin{proof}
$\NP$-hardness was proved in Corollary \ref{co:spherical-NP-hard}, and the conditions in Proposition \ref{pr:spherical-solution} are easily seen to furnish a certificate, proving that the problem is in $\NP$. The solution size bound of Theorem \ref{th:solution-bound} was also established in Proposition \ref{pr:spherical-solution}.
\end{proof}

\subsection{Orientable equations of genus $g\ge 1$}

Consider an orientable equation 
\begin{equation}\label{eq:orientable-eq}
W = \prod_{j=1}^g [x_j,y_j] \prod_{i=1}^k z_i^{-1}c_iz_i = 1 
\end{equation}
of genus $g \geq 1$,
with coefficients $c_i\in\BS(1,n)$ and variables $x_j,y_j,z_i$. 

\begin{proposition}\label{prop:comm-eq}
The equation \eqref{eq:orientable-eq} has a solution
if and only if 
\begin{equation}\label{eq:comm-eq}
\sigma_t \rb{\prod_{i=1}^k c_i}=0
\ \ \mbox{ and }\ \ 
\sigma_a \rb{\prod_{i=1}^k c_i}
\mbox{ is a multiple of } n-1.
\end{equation}
Therefore, the Diophantine 
problem for \eqref{eq:orientable-eq}
is decidable in linear time.
\end{proposition}

\begin{proof}
It is easy to see that the conditions of Lemma 
\ref{le:derived-sbgp-conditions} are satisfied by $\prod_{i=1}^k c_i$ if and only if they are satisfied by $\prod_{i=1}^k z_i^{-1}c_iz_i$ (for any choice of $z_i$). Hence, 
it follows from Lemma \ref{le:comwidth} that
$W=1$ has a solution if and only if
$\prod_{i=1}^k c_i$ belongs to the derived subgroup, which is true if and only if the conditions
\eqref{eq:comm-eq} hold.
These conditions can obviously be checked in linear time.
\end{proof}

\begin{proof}[Proof of theorem \ref{th:solution-bound} for orientable equations of genus $g\ge 1$]

If $W = 1$ is an equation of the form \eqref{eq:orientable-eq}
that has a solution, then $\prod_{i=1}^k c_i
\in 
\BS(1,n)'$, and we have $\prod_{i=1}^k c_i = t^{p}a^{K(n-1)}t^{-p}$ for some $p,K\in \MZ$ satisfying bounds
$$|K(n-1)|<|W||n|^{|W|} \ \ \ \mbox{and} \ \ \ 0 \leq p \leq |W|$$ by Lemma \ref{le:standard-word}. This yields a solution with $x_1 = t$, $y_1 = a^{-K}t^{-p}$ and all other variables equal to 1. From the bounds on $|K|$ and $p$, 
it follows $a^{-K}t^{-p}$ can be expressed as a word of length $O(|W|)$,
completing the proof for this case.
\end{proof}

\subsection{Non-orientable equations}\label{subsec:nonorientable}

In this section, we prove that the Diophantine problem 
for non-orientable quadratic
equations can be solved in linear time
for genus $g\ge 2$, and that it is $\NP$-complete for genus $g=1$ whenever $|n|>1$. We first consider the genus 1 case.
\subsubsection{Genus $g=1$}
The following is a nonorientable version of Proposition \ref{pr:spherical-solution}.

\begin{proposition}\label{pr:non-orientable-g1-solution}
Let $c_i=(\alpha_i,\beta_i)\in \MZ[\tfrac{1}{n}] \rtimes \MZ$ \  for \ $i=1,\ldots,k$, let $\beta = \gcd (|\beta_1|,\ldots,|\beta_k|)$ and let $\beta_x = -\frac{1}{2}\sum_{i=1}^k \beta_i$.
Then the equation
\begin{equation}\label{eq:le-nonorient}
x^2\prod_{i=1}^k z_i^{-1}c_iz_i = 1
\end{equation}
has a solution if and only if the following conditions hold
\begin{enumerate}[(i.)]
    \item \label{cond:nonorient-beta}
    $\sum_{i=1}^k\beta_i \equiv 0 \ (\mbox{mod }2)$

    \item \label{cond:nonorient-alpha}
    there exist $x_1,
    \ldots,x_k\in \MZ$ such that $\alpha_1n^{x_1}+\cdots + \alpha_k n^{x_k} \equiv 0 \ (\mbox{mod }K)$,
\end{enumerate} 
 where $K = \gcd(n^{|\beta_x|}+1,n^{\beta}-1)$. Moreover, if \eqref{eq:le-nonorient} has a solution, then it has a solution that is $O(|W|^3)$.
\end{proposition}

\begin{proof} Suppose $x=(\alpha_x,\beta_x)$, $z_i=(v_i,y_i)$ is a solution to \eqref{eq:le-nonorient}. As in the proof of Lemma \ref{pr:spherical-solution},
we have by direct computation 
\begin{equation*}\label{eq:alpha-beta-nonorient}
\rb{
\alpha_x (n^{-\beta_x}+1)+
n^{-2\beta_x}\sum_{i=1}^k n^{-\sum_{j=1}^{i-1}\beta_j} 
n^{y_i}(\alpha_i+v_i(n^{-\beta_i}-1))
,2\beta_x+\sum_{i=1}^k\beta_i} = (0,0),
\end{equation*}
which implies the necessity of condition \eqref{cond:nonorient-beta} and shows that $\beta_x=-\frac{1}{2}\sum_{i=1}^k \beta_i$. Substituting into the first component, we obtain

\begin{equation*}\label{eq:alpha-beta-comp1}
\alpha_x (n^{-\beta_x}+1)+
\sum_{i=1}^k n^{\sum_{j=i}^k\beta_j} 
n^{y_i}(\alpha_i+v_i(n^{-\beta_i}-1))
 = 0.
\end{equation*}
Let $v_i'$ be defined as in Proposition \ref{pr:spherical-solution}, let
$$
\alpha_x' = \begin{cases}
    -n^{-\beta_x}\alpha_x 
    \mbox{ if }\beta_x > 0 \\
    -\alpha_x 
    \mbox{ otherwise}
    \end{cases}
$$
and let $L$ be chosen so that $n^L\alpha_x'\in \MZ$ and $n^{L+y_i+\sum_{j=i}^{k}\beta_j}v_i' \in \MZ$ for $i = 1,\ldots,k$. Then, setting $x_i = L + y_i + \sum_{j=i}^{k}\beta_j$, we have
\begin{equation}\label{eq:bigL}
\alpha_1n^{x_1}+\cdots + \alpha_k n^{x_k} = n^L\alpha_x'(n^{|\beta_x|}+1)+\sum_{i=1}^k n^{x_i}v_i'(n^{|\beta_i|}-1).
\end{equation}
By Lemma \ref{le:easy-fact}, this shows that \eqref{cond:nonorient-alpha} holds.

Conversely, suppose that \eqref{beta-cond} and \eqref{alpha-cond} hold, so there are integers $x_1,
\ldots,x_k$ such that 
\begin{equation}\label{eq:modulus}
 \alpha_1n^{x_1}+\cdots + \alpha_k n^{x_k} \equiv 0 \ (\mbox{mod } K)
\end{equation}
where $K = \gcd(n^{|\beta_x|}+1, n^{\beta}-1).$ We consider the following two cases.
\setcounter{case}{0}

\begin{case}
Suppose that $\beta_1=\cdots=\beta_k=0$ (so that $\beta=\beta_x=0$). Then $K=n^{|\beta_x|}+1 = 2$, and we may solve \eqref{eq:bigL} by setting $\alpha_x' = \frac{1}{2}\sum_{i=1}^k \alpha_i n^{x_i}$ and $v_1' = \cdots = v_k' = 0$. Following the proof of Proposition \ref{pr:spherical-solution} we obtain, in this way, a solution that is $O(|W|^3)$.
\end{case}
\begin{case}
Suppose that at least one of the $
\beta_i$ is nonzero, so that $\beta \neq 0$. Then by Lemmas \ref{le:bezout} and \ref{le:easy-fact}, there are $s_x, s_1,\ldots,s_k, h\in \MZ$ such that 
\begin{equation*}
\alpha_1 n^{x_1}+
\cdots + \alpha_k n^{x_k} = h\rb{s_x(n^{\beta_x}+1)+\sum_{i=1}^k s_i (n^{|\beta_i|}-1)}
\end{equation*}
with $|s_x|,|s_i|< |n|^{|\beta_x|}+1+\sum_{i=1}^k (|n|^{|\beta_i|}+1)< k+1+\prod_{i=1}^k |n|^{|\beta_i|} \leq |W| + |n|^{|W|}$. We obtain the same bound for $h$ as in Case 2 of the proof of Proposition \ref{pr:spherical-solution}, and a similar argument yields a solution that is $O(|W|^2)$. 
\end{case}

\end{proof}

\begin{corollary}\label{co:meta-nonorient-complete}
The Diophantine problem for nonorientable equations of genus 1
\begin{equation*}
W = x^2\prod_{i=1}^k z_i^{-1}c_iz_i = 1
\end{equation*}
over $\BS(1,n)$ is $\NP$-complete. Moreover, this class of equations satisfies Theorem \ref{th:solution-bound}.
\end{corollary}

\subsubsection{Genus $g\ge2$}

\begin{proposition}\label{pr:2squares}
The nonorientable equation 
\begin{equation}\label{eq:nonorientable-genus2}
W = x_1^2\cdots x_g^2\prod_{i=1}^k z_i^{-1}c_iz_i = 1
\end{equation}
of genus $g\ge 2$
with unknowns $x_i,z_i$ has a solution in $\BS(1,n)$ if and only if 
the following holds
\begin{equation}\label{eq:nonorientable-genus2-condition}
\sigma_t\rb{\prod_{i=1}^k c_i} \mbox{ is even }
\ \wedge\ 
\left[
n \mbox{ is even }
\ \vee\ 
\sigma_a\rb{\prod_{i=1}^k c_i} \mbox{ is even }
\right].
\end{equation}
\end{proposition}

\begin{proof} 
Let $S=\gp{x^2 \mid x\in\BS(1,n)}$ be the verbal subgroup generated by squares. Since $S$ contains the derived subgroup $\BS(1,n)'$, it follows that $\prod_{i=1}^k z_i^{-1}c_iz_i \in S$ if and only if $\prod_{i=1}^k c_i \in S$. Hence, by Lemma \ref{le:square-verbal-width}, $W=1$ has a solution if and only if $\prod_{i=1}^k c_i \in S$. Following \eqref{eq:verbwidth}, this is true if and only if the conditions \eqref{eq:nonorientable-genus2-condition} hold.
\end{proof}
\begin{corollary}\label{co:genus-geq2}
    The Diophantine problem for nonorientable equations of genus $g \geq 2$ in $\BS(1,n)$ is decidable in linear time. If a solution exists, then there is a solution of size $O(|W|)$.
\end{corollary}
Note that Theorem \ref{th:meta-general-complete} follows from Corollaries \ref{co:spherical-NP-hard}, \ref{th:meta-spherical-complete}, \ref{co:meta-nonorient-complete} and \ref{co:genus-geq2}, together with Proposition \ref{prop:comm-eq} and Theorem \ref{th:gen-np-hard}.

\bibliography{main_bibliography}

\end{document}